\documentclass{amsart}
\usepackage{amssymb}
\usepackage{graphicx}
\usepackage{color}
\usepackage[shortlabels]{enumitem}
\usepackage[T1]{fontenc}
\usepackage{amsmath,color}
\usepackage{hyperref}
\usepackage{url}
\usepackage{longtable}

\newtheorem{thm}{Theorem}[section]
\newtheorem{prp}[thm]{Proposition}
\newtheorem{lem}[thm]{Lemma}
\newtheorem{cor}[thm]{Corollary}
\newtheorem{rem}[thm]{Remark}
\newtheorem{ques}[thm]{Question}

{

\def\tr{\mathop{\rm tr }\nolimits}

\title{A note on graphs with exactly two main eigenvalues}
\author[S. Hayat]{ Sakander Hayat$^\diamondsuit$ }
\thanks{$^\diamondsuit$S.H is supported by a CAS-TWAS president's fellowship at USTC, China.}
\author[J. H. Koolen]{ Jack H. Koolen$^\spadesuit$ }
\thanks{$^\spadesuit$J.H.K. is partially supported by the National Natural Science Foundation of China (No. 11471009).}
\author[F. Liu]{Fenjin Liu$^\clubsuit$}
\thanks{$^\clubsuit$F.L is supported by NSFC grant No. 11401044 and Postdoctoral Science Foundation of China No. 2014M560754.}
\author[Z. Qiao]{Zhi Qiao}

\address{School of Mathematical Sciences,
University of Science and Technology of China,
Hefei, Anhui, 230026, P.R. China}
\email{sakander@mail.ustc.edu.cn}
\address{Wen-Tsun Wu Key Laboratory of CAS and School of Mathematical Sciences,
University of Science and Technology of China,
Hefei, Anhui, 230026, P.R. China}
\email{koolen@ustc.edu.cn}

\address{School of Science, Chang'an University, Xi'an, Shaanxi 710064, P.R. China}
\email{fenjinliu@yahoo.com}

\address{College of Mathematics and Software Science,
Sichuan Normal University, Chengdu, 610068, Sichuan, P.R. China}
\email{zhiqiao@sicnu.edu.cn}

\begin{document}




\begin{abstract}
In this note, we consider connected graphs with exactly two main eigenvalues.
We will give several constructions for them, and as a consequence we show a family
of those graphs with an unbounded number of distinct valencies.
\end{abstract}

\maketitle

\section{Introduction}\label{sec1}

Let $G$ be a simple graph with vertex set $V$
and $(0,1)$-adjacency matrix $A$. By an \emph{eigenvalue} (resp. \emph{eigenvector}) of a graph, we mean an
eigenvalue (resp. eigenvector) of its adjacency matrix $A$. The largest eigenvalue of $G$ is
called the \emph{spectral radius} of $G$ and denoted by $\rho$.
We denote the all-ones matrix, the identity
matrix and the all-one vector by $J$, $I$ and $\mathbf{j}$ respectively.
An eigenvalue $\mu$ of $A$ is  said to be a \emph{main eigenvalue} of $A$,
if there exists an eigenvector of $A$ corresponding to $\mu$
not orthogonal to $\mathbf{j}$.
Also, $\mu$ is called \emph{non-main}, if for any eigenvector $\mathbf{x}$ of
$A$ corresponding to $\mu$, $\mathbf{x}$ is orthogonal to $\mathbf{j}$.
Note that, for a connected graph $G$, the spectral radius $\rho$ is always
a main eigenvalue by the Perron-Frobenius Theorem
(See \cite[Theorem 8.8.1]{gr01}). We refer the reader to the survey on
main eigenvalues of graphs by Rowlinson \cite{R07}.\\

A long-standing problem posed by Cvetkovi\'{c} \cite{C78} is
to characterize graphs with exactly $k~(k\geq2)$ main eigenvalues.
Note that a graph $G$ has exactly one main eigenvalue if and only if
$G$ is regular \cite{C78}. Hagos \cite{H02} gave a characterization
of graphs with exactly $k$ main eigenvalues. He showed:
\begin{thm}\cite[Theorem 2.1]{H02}\label{thm-hagos}
Let $G$ be a graph and $k\geq1$ be maximal such that
$\mathbf{j},~A\mathbf{j},\ldots,A^{k-1}\mathbf{j}$ are linearly independent.
Then $G$ has exactly $k$ main eigenvalues.
\end{thm}
Now we turn to the case where $G$ has exactly two main eigenvalues.
Theorem \ref{thm-hagos} states that $\mathbf{j}$ and $A\mathbf{j}$
are linearly independent and that $\mathbf{j},~A\mathbf{j}$ and $A^{2}\mathbf{j}$
are linearly dependent. If we denote $\mathbf{d}:=A\mathbf{j}$, then this
means that there are real numbers $\alpha,~\beta$ such that
\begin{equation*}
A\mathbf{d}=\alpha\mathbf{d}+\beta\mathbf{j},
\end{equation*}
and $G$ is not regular.
Such a graph $G$ is called 2-\emph{walk $(\alpha, \beta)$-linear} in \cite{HT06}.
Hagos \cite{H02} showed the following relation.
\begin{prp}\cite[Corollary 2.5]{H02}\label{hagos-cor}
Let $G$ be a connected \emph{2}-walk $(\alpha, \beta)$-linear graph.
Then the two main eigenvalues of $G$ are $\mu_{0}$ and $\mu_{1}$, where
\begin{equation*}
\mu_{0},\mu_{1}=\frac{\alpha\pm\sqrt{\alpha^{2}+4\beta}}{2}.
\end{equation*}
\end{prp}
In view of the above definition, the following questions arise:
\begin{ques}\label{que}
Let $\alpha,~\beta$ be integers such that $\alpha\geq0$. Then
\begin{itemize}
\item[\emph{(i)}] For which $(\alpha,\beta)$ does there exist a connected \emph{2}-walk $(\alpha,\beta)$-linear graph?
\item[\emph{(ii)}] How many different valencies can a \emph{2}-walk $(\alpha,\beta)$-linear graph have?
\item[\emph{(iii)}] How many connected non-isomorphic \emph{2}-walk $(\alpha,\beta)$-linear graphs exist for given $\alpha,\beta$?
\end{itemize}
\end{ques}

In view of Question \ref{que}(i), Lin and Qiongxiang \cite{LQ14} showed the following existence result.
\begin{prp}\cite{LQ14}\label{LQresult}
For integers $\alpha\geq0$ and $\beta$, a \emph{2}-walk $(\alpha,\beta)$-linear graph exists if and only if
$\alpha^{2}+4\beta\geq4$ and $(\alpha,\beta)\neq(0,1)$.
\end{prp}
In Section \ref{sec2}, we will give constructions of 2-walk $(\alpha, \beta)$-linear graphs for all pairs $(\alpha, \beta)$
satisfying the conditions of Proposition \ref{LQresult}.
Regarding Questions \ref{que}(ii)-(iii), we will show that if for fixed $\alpha$, $\beta$, there exists one connected
2-walk $(\alpha, \beta)$-linear graph, which is not a tree, then there exist infinitely many
pairwise non-isomorphic 2-walk $(\alpha, \beta)$ linear graphs.
Also, we will show that connected
non-bipartite 2-walk $(\alpha, \beta)$-linear graphs exist with exactly two valencies for infinitely
many values of $(\alpha, \beta)$. This shows that the question of Cvetkovi\'{c} to characterize
graphs with two main eigenvalues is very difficult.\\

Now we will discuss harmonic graphs.
We call a graph $\delta$-\emph{harmonic}, if $A\mathbf{d}=\delta\mathbf{d}$ holds.
Note that, connected $\delta$-harmonic graphs with $\delta=0,1$ are regular.
A graph $G$ is \emph{harmonic} if it is $\delta$-harmonic for some positive integer $\delta$.
Note that a regular graph is harmonic. Also, the disjoint union of a harmonic graph and an
isolated vertex is harmonic.
The harmonic trees were determined by Gr\"{u}newald \cite{G02}. He showed that they are exactly the complete graphs
$K_{1}$, $K_{2}$ and $T_{\lambda}$ where $\lambda\geq2$ is an integer and $T_{\lambda}$ is the tree
with one vertex $v$ of valency $\lambda^{2}-\lambda+1$, and any
neighbour of $v$ has valency $\lambda$ and the remaining vertices have valency 1.
The tree $T_{\lambda}$ is $\lambda$-harmonic with $\lambda\geq2$.
Note that non-regular connected harmonic graphs are exactly the
2-walk $(\alpha,0)$-linear graphs.
The following result, which follows from Proposition \ref{hagos-cor},
gives an alternative characterization of non-regular
harmonic graphs in terms of their main eigenvalues.
\begin{prp}\emph{(Cf. \cite[Theorem 8]{N06}.)}\label{thm3}
Let $G$ be a connected graph with spectral radius
$\rho>0$. Then $G$ is harmonic and non-regular if and only if
the only main eigenvalues of $G$ are $\rho$ and $0$.
\end{prp}
Non-regular 2-walk $(\alpha,\beta)$-linear trees are characterized in
\cite{HZ05}, whereas non-regular connected
2-walk $(\alpha,\beta)$-linear graphs with
a small number of cycles are studied in \cite{HT06}, \cite{HLZ05},
\cite{FLG05} and \cite{BGGP03}.\\

The following result characterizes the connected 2-harmonic graphs, which is
shown by Gr\"{u}newald \cite{G02}.
\begin{prp}\emph{(Cf. \cite[Corollary 2.1]{G02})}\label{delta-harmonic-trees}
Let $G$ be a connected \emph{2}-harmonic graph. Then $G$ is either a cycle or the tree $T_{2}$.
\end{prp}


\section{Equitable graphs}\label{sec2}
In this section, we will construct graphs with at most $r$
main eigenvalues by using equitable partitions of graphs. We first recall
some definitions which will be used in this section.\\

A graph $G$ is called \emph{$t$-valenced} if it has
exactly $t$ distinct valencies $k_{1},\ldots,k_{t}$.
We call $G$ a \emph{biregular graph} when $t=2$.
Assume that $G$ has $t$ distinct valencies
$k_{1},\ldots,k_{t}$. We write $V_i:=\{v\in V (G)\mid d_v=k_i\}$
and $n_i:=|V_i|$ for $i\in\{1,\ldots,t\}$. Clearly the
subsets $V_i$ partition the vertex set of $G$ and this partition
is called the \emph{valency partition} of $G$.
Let $\pi = \{\pi_{1},\ldots,\pi_{t}\}$ be a partition of the vertices of $G$,
where $\pi_{i}$ is called a \emph{block} of $\pi$.
For each vertex $x$ in $\pi_{i}$, write $d^{(j)}_{x}$
for the number of neighbours of $x$ in $\pi_{j}$. Then
we write $b_{ij} = \frac{1}{\mid\pi_{i}\mid} \sum\limits_{x\in \pi_{i}} d^{(j)}_{x}$ for
the average number of neighbours in $\pi_{j}$ of vertices in $\pi_{i}$.
The matrix $B_{\pi}:=(b_{ij})$ is called the \emph{quotient matrix} of $\pi$
and $\pi$ is called \emph{equitable} if for all $i$ and $j$, we have $d^{(j)}_{x} = b_{ij}$
for each $x\in \pi_{i}$. A graph $G$ is called an \emph{equitable graph} if
its valency partition is equitable.
An \emph{equitable biregular graph} is a biregular graph whose valency partition
is equitable.
We refer the reader to Godsil and Royles's book
\cite[Chapter 9]{gr01} for the necessary background on equitable
partitions.\\

The following result was essentially shown by Cvetkovi\'{c} \cite{C78}.
It gives a sufficient condition for a graph to have at most $r$ main eigenvalues.
\begin{thm}\emph{(Cf. \cite[Theorem 3]{C78}.)}\label{eqtble-prop}
Let $G$ be a connected graph and $\pi$ be an equitable partition of $G$.
Let $Q$ be the quotient matrix of $\pi$, say with exactly $r$ distinct eigenvalues.
Then
\begin{itemize}
\item[\emph{(1)}]$G$ has at most $r$ main eigenvalues.
\end{itemize}
In particular, the following holds:
\begin{itemize}
\item[\emph{(a)}] If $G$ is equitable and $t$-valenced, then $G$ has at most $t$ main eigenvalues;
\item[\emph{(b)}] If, moreover $G$ is equitable and biregular, then $G$ has exactly two main eigenvalues.
\end{itemize}
\end{thm}

\begin{proof}
Assume $\pi = \{\pi_{1},\ldots, \pi_{t}\}$. Let $P$ be the $|V(G)|\times t$ matrix with
characteristic vectors $(\chi(p_i))$ of blocks of $\pi$ as its columns.
Let $\sigma_1,\sigma_2\ldots,\sigma_{t}$ be the distinct eigenvalues of $Q$.
Then $\mathbf{j}'=\sum\limits_{i=1}^{r}\mathbf{v}_{i}$, where $\mathbf{j}'$ is the all-ones vector of
length $t$ and $\mathbf{v}_i$ is an eigenvector of $Q$ corresponding to $\sigma_{i}$ $(1\leq i\leq r)$.
And $\mathbf{j}=P\mathbf{j}'=\sum\limits_{i=1}^{r}\mathbf{w}_{i}$ where $\mathbf{j}$ is the all-ones vector
of length $|V(G)|$ and $\mathbf{w}_{i}=P\mathbf{v}_{i}$ is an eigenvector of $A$ corresponding to
$\sigma_{i}$ $(1\leq i\leq r)$. Thus $G$ has at most $r$ main eigenvalues.

(1) implies (a) as the valency partition is equitable and its quotient matrix has at most $t$ distinct eigenvalues.
Moreover, if $G$ is
equitable biregular, then the number of main eigenvalues can not be one.
Thus, $G$ has precisely two main eigenvalues.
\end{proof}

%

\begin{rem}
For $\lambda\geq2$, the tree $T_{\lambda}$ (as introduced in Section \emph{\ref{sec1}}) is equitable
and has three distinct valencies, but only two main eigenvalues. This shows that you can have less
than $r$ main eigenvalues in Proposition \emph{\ref{eqtble-prop}}.
\end{rem}

\begin{rem}\label{rem1}
The condition \emph{(b)} of Proposition \emph{\ref{eqtble-prop}} gives a recipe to construct infinitely many
\emph{2}-walk $(\alpha,\beta)$-linear biregular graphs. A particular instance is to consider the
cone\footnote{The \emph{cone over a graph} $G$ is the graph with vertex set $\{ \infty \} \cup V(G)$
such that $\infty$ is adjacent to all vertices of $G$. When $G$ is not specified, we call it a \emph{cone}.}
over a regular graph. Note that the cone of a regular graph $G$ is harmonic if the graph $G$ is complete
i.e. which is also a complete graph.
As $\alpha$ and $\beta$ are determined by the main eigenvalues $\mu_{0}=\rho$ and $\mu_{1}$ uniquely for a connected
\emph{2}-walk $(\alpha,\beta)$-linear graph, it is easy to construct infinitely many values of $(\alpha,\beta)$
such that a \emph{2}-walk $(\alpha,\beta)$-linear graph exists, using the cone over regular graphs.
\end{rem}


In the spirit of Proposition \ref{LQresult}, we will determine the pairs $(\alpha, \beta)$ for which there exists a connected
equitable biregular 2-walk $(\alpha, \beta)$-linear graph.
\begin{thm}
Let $\alpha, \beta$ be integers such that $\alpha \geq 0$.
Then there exists a connected equitable biregular \emph{2}-walk $(\alpha, \beta)$-linear graph
if and only if $\alpha^2 + 4 \beta > 4$.
\end{thm}
\begin{proof}
Let $\alpha, \beta$ be integers such that $\alpha \geq 0$ and $\alpha^2 + 4 \beta > 4$.
To show the existence in view of Theorem \ref{eqtble-prop} and Remark \ref{rem1}, we need
to show the existence of a $2 \times 2$ matrix
                $Q=\left(
                      \begin{array}{cc}
                        q_{11} & q_{12} \\
                        q_{12} & q_{22} \\
                      \end{array}
                    \right)$
such that all entries are non-negative integers, $q_{12}$, $q_{21}$ are positive,
$\tr(Q) = \alpha$, $\det(Q) = -\beta$ and $q_{11} + q_{12} < q_{21} + q_{22}$.
Let $\alpha' := \lfloor \frac{\alpha}{2} \rfloor$.
Let $q_{11} := \alpha'$, $q_{22} := \alpha - \alpha'$, $q_{12} = 1$ and
$q_{21} := -\det{Q} + \alpha' (\alpha- \alpha')$.
This gives us a matrix $Q$ as required.
Thus, it shows the existence of the required graphs.

On the other hand in view of Proposition \ref{LQresult}, we only need to consider the case
where $\alpha^2 + 4 \beta = 4$. Let $Q$ be a $2\times 2$ matrix such that all
entries are non-negative integers, $q_{12}$, $q_{21}$ are positive, $\tr(Q) = \alpha$ and $\det(Q) =-\beta = (\frac{\alpha}{2})^2 -1$.
Then $q_{11}=\alpha/2=q_{22}$ and $q_{12}=q_{21}=1$ and hence $q_{11}+q_{12}=q_{21}+q_{22}=\frac{\alpha}{2}+1$.
So for this case there does not exist a connected equitable biregular
2-walk $(\alpha,\beta)$-linear graph. This shows the theorem.
\end{proof}


\begin{rem}
There exists a connected \emph{3}-valenced equitable \emph{2}-walk $(\alpha, \beta)$-linear graph with $\alpha^2 + 4\beta = 4$,
by considering the connected equitable graphs with quotient matrix ($\alpha$ is even in this case, and let $\alpha' = \frac{\alpha}{2}$)
$\left(
  \begin{array}{ccc}
    \alpha'-1 & 1 & 0 \\
    1 & \alpha'-1 & 1 \\
    0 & 3 & \alpha'-1 \\
  \end{array}
\right)$.
This shows that for the possible pairs of $(\alpha, \beta)$ for which there exists a
\emph{2}-walk $(\alpha, \beta)$-linear graph (see Proposition \emph{\ref{LQresult}}),
there exists an equitable \emph{2}-walk $(\alpha, \beta)$-linear graph.
\end{rem}

\section{Seidel matrix, switching classes and regular two-graphs}\label{sec3}
In this section, we study the Seidel matrix of a graph, switching classes and regular two-graphs.\\

We call a regular graph \emph{strongly regular} if there are constants
$\lambda$ and $\mu$ such that every pair of distinct vertices has $\lambda$
or $\mu$ common neighbours if they are adjacent or non-adjacent respectively.
The \emph{Seidel matrix} $S$ of a graph, with adjacency matrix $A$, is defined by $S=J-I-2A$.
A \emph{strong graph} is a graph such that its Seidel matrix $S$ satisfies $S^{2}\in\langle S,I,J\rangle$,
where $\langle \ldots\rangle$ denotes the $\mathbb{R}$-span. Seidel \cite{Sei68} showed that:
\begin{prp}\cite{Sei68}
Let $G$ be a graph with Seidel matrix $S$. Then $G$ is strong if and only if at least one of the following
holds:
\begin{itemize}
\item[\emph{(i)}] $G$ is strongly regular.
\item[\emph{(ii)}] $S$ has exactly two distinct eigenvalues.
\end{itemize}
\end{prp}

Let $\sigma=\{U,V-U\}$ be a bipartition of the vertex set $V$ of $G$.
The graph $G^{\sigma}$ with vertex set $V$ is obtained from $G$ as follows:
if $x,~y$ are two distinct vertices of $V$ such that if $x,y\in U$ or $x,y\in V-U$,
then $x\sim y$ in $G^{\sigma}$ if and only if $x\sim y$ in $G$ and if
$\mid\{x,y\}\cap U\mid=1$, then $x\nsim y$ in $G^{\sigma}$ if and only if
$x\sim y$ in $G$. In other words, the edges and non-edges between $U$ and
$V-U$ have been switched. We call $G^{\sigma}$ the graph obtained from $G$ by
\emph{(Seidel) switching} with respect to $\sigma$. It is well-known that the spectra of
$S(G)$ and $S(G^{\sigma})$ are the same. The \emph{switching class} or
\emph{two-graph} $[G]$ of $G$ is the set \{$G^{\sigma}$ $\mid$ $\sigma$ is a bipartition
of the vertex set of $G$ possible with one part empty\}.\\

Note that switching induces an equivalence relation on graphs, with switching classes
as its equivalence classes. We say that the two-graph $[G]$ is \emph{regular} if the Seidel matrix $S(G)$ has exactly
two distinct eigenvalues. Note that if the number of vertices of $G$ is at least two, then the
Seidel matrix $S(G)$ has at least two eigenvalues. The regular two-graphs containing a
complete graph or an empty graph are called \emph{trivial}.
The graphs in regular two-graphs are examples of strong graphs.
We refer to the detailed survey on two-graphs by Seidel \cite{Sei76}.\\

The following lemma is a slight modification of \cite[Lemma 4.1, Proposition 4.2, Proposition 4.3]{VKX15} by Van Dam et al.
\begin{lem}\cite{VKX15}\label{thm1}
Let $G$ be an $n$-vertex graph within a non-trivial regular two-graph. Then the following hold.
\begin{itemize}
  \item[\emph{(i)}] If $G$ is regular, then it is a strongly regular.
  \item[\emph{(ii)}] If $G$ is not connected, then it is a disjoint union of an isolated vertex and a strongly regular graph.
  \item[\emph{(iii)}] If $G$ is non-regular with Seidel eigenvalues $[-1-2\theta_{0}]^{m_{0}},[-1-2\theta_{1}]^{m_{1}}$,
  then it has four distinct (adjacency) eigenvalues $\mu_{0},\mu_{1},[\theta_{0}]^{m_{0}-1},[\theta_{1}]^{m_{1}-1}$
     of which $\mu_{0}$ and $\mu_{1}$ are main eigenvalues and $\theta_{0},\theta_{1}$ are non-main eigenvalues. Main eigenvalues
      $\mu_{0},\mu_{1}$ are uniquely determined by $\theta_{0},~\theta_{1},~m_{0},~m_{1},~n$ and the number of edges.
\end{itemize}
\end{lem}

Let $V=GF(2)^{2r}$ endowed with a non-degenerate symplectic bilinear form, where $r$
is a positive integer.
Let $G$ be the graph with vertex set $V$ and $u\sim v$
if $\langle u,v\rangle\neq 0$. The switching class $[G]$ is known as the symplectic
two-graph; and it is regular with the two Seidel eigenvalues $\pm2^{r-1}$. It is clear that
graph $G$ has $0$ as an isolated vertex. The other component of $G$ is known as the
symplectic graph $Sp(2r)$, which is strongly regular with parameters
$(2^{2r}-1,2^{2r-1},2^{2r-2},2^{2r-2})$.
In Table \ref{tab1}, which is slight modification of \cite[Table 5.1]{VKX15}, all non-regular graphs in the
switching classes of symplectic two-graph on
16 vertices are listed. Also, the possible values of $(\alpha,\beta)$ (first column)
for which it is 2-walk $(\alpha,\beta)$-linear graph,
main eigenvalues (second column), sequences
of valencies (third column), and how many times each of these occur (last column) are listed.
Note that the same sequence of valencies can be shared by non-isomorphic graphs.\\

Van Dam et al. \cite{VKX15} used a result of Vu \cite{V96} to show the following result.
\begin{thm}\cite[Theorem 5.1]{VKX15}\label{thm1a}
Let $t\geq3$, and $H$ be a graph on $n$ vertices with $t$ distinct valencies. Then there exists a connected
graph $G$ on at most $2^{n+2}$ vertices with four distinct eigenvalues and at least $t$ distinct valencies,
having $H$ as an induced subgraph.
\end{thm}

 In \cite{VKX15}, the graph $G$ in the above theorem is taken such that the switching class $[G]$ is the symplectic regular two-graph coming from  $Sp(2r)$, with $r = \lceil\frac{n+1}{2}\rceil$. This shows that  the graph $G$ has exactly two main eigenvalues, and hence we have:
\begin{thm}\label{thm2}
Let $t\geq3$, and $H$ be a graph on $n$ vertices with $t$ distinct valencies. Then there exists a connected
graph $G$ on at most $2^{n+2}$ vertices with exactly two main eigenvalues and at least $t$ distinct valencies,
having $H$ as an induced subgraph.
\end{thm}

The following is an immediate consequence of Theorem \ref{thm2} and the definition of 2-walk $(\alpha,\beta)$-linear graph.
This gives an answer to the Question \ref{que}(ii) posed in Section \ref{sec1}.
\begin{cor}
For any positive integer $t\geq1$, there exists a \emph{2}-walk $(\alpha,\beta)$-linear graph $G$ having at least $t$ different
valencies.
\end{cor}
As we already showed that for every integer $\delta\geq3$, there exists a connected equitable biregular $\delta$-harmonic
graph and for every $t\geq1$, there exists a connected 2-walk $(\alpha,\beta)$-linear graph with at least $t$ valencies,
we wonder wether the same is true for harmonic graphs.
Therefore, we would like to ask the following question:
\begin{ques}
For every $t$, does there always exist a harmonic graph with at least $t$ valencies?
\end{ques}
We think this is true. Note that a graph $G$, whose switching class $[G]$, is a non-trivial regular two-graph and has the same number of edges
as a graph $H$, in $[G]$, having an isolated vertex as one of its connected components, is always harmonic.
In the symplectic two-graph on 16 vertices, there are 8-harmonic graphs with 16 vertices and 4 distinct valencies,
see Table \ref{tab1}.

\section{Constructing infinite families}\label{sec4}
In this section, we show that if for fixed $\alpha$, $\beta$, there exists one connected
2-walk $(\alpha, \beta)$-linear graph, which is not a tree, then there exist infinitely many
pairwise non-isomorphic 2-walk $(\alpha, \beta)$ linear graphs.
As a consequence, we show that there exist
infinitely many non-isomorphic harmonic graphs.\\

Let $G$ and $H$ be two non-regular 2-walk $(\alpha,\beta)$-linear graphs.
Assume that $G$ and $H$ have respective edges $e=xy$ and $f=uv$, where $G-e$ and $H-f$ are connected, such
that $d_{x}=d_{u}$ and $d_{y}=d_{v}$.
Construct a graph $L$ from the disjoint union of $G$ and $H$ by replacing
the edges $xy$ and $uv$ by $xv$ and $yu$ respectively.
Then $L$ is connected, non-regular and 2-walk $(\alpha,\beta)$-linear (see F{\tiny IGURE} \ref{cons}).
We can start with $G\cong H$, $x=u$ and $y=v$ to obtain an infinite family of connected non-regular 2-walk
$(\alpha,\beta)$-linear graphs. Note that this infinite family has unbounded diameter as the maximal
valency does not change.
\begin{figure}[h!]
\centering
  \includegraphics[width=7cm]{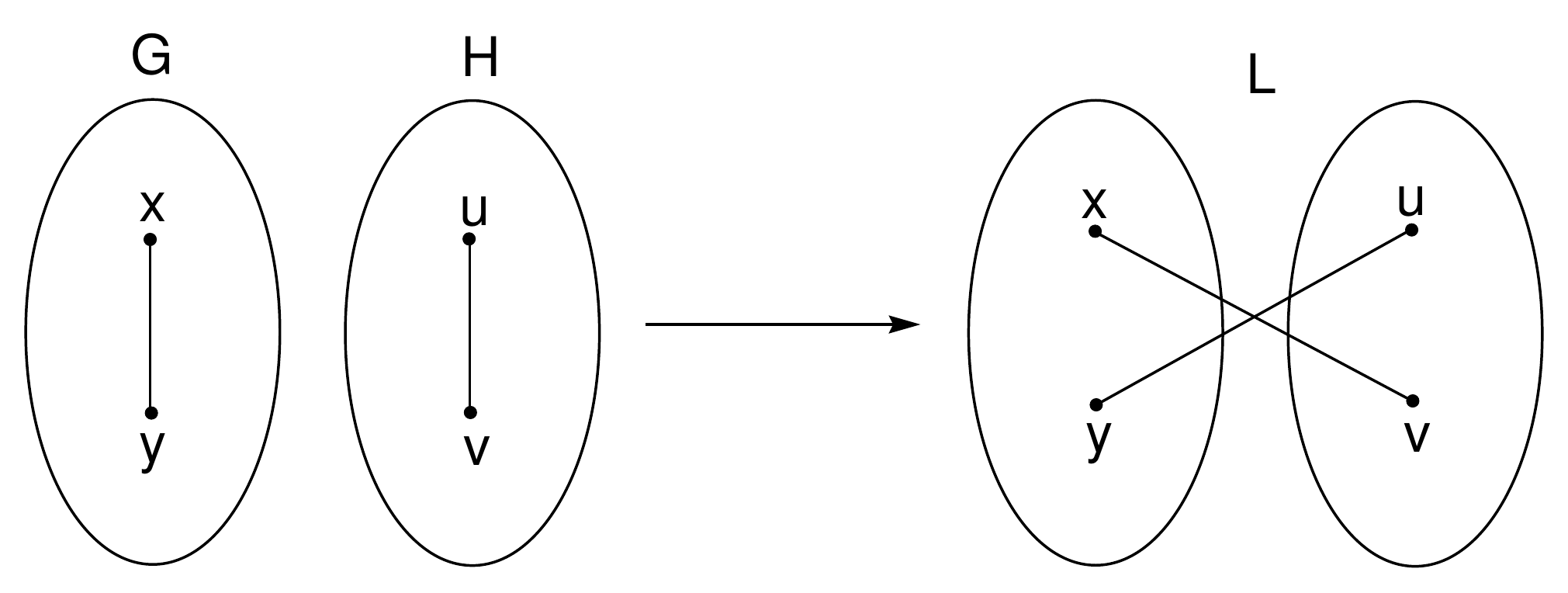}\\
  \caption{}\label{cons}
\end{figure}

Similar operation can be used to obtain an infinite family of harmonic graphs. If you start with
a harmonic graph and by using operation explained in F{\tiny IGURE} \ref{cons}, we can construct an infinite
family of harmonic graphs. It is worthy to remark that this construction can be easily modified to
construct many of these graphs.\\

The following result shows that if there exists a connected 2-walk $(\alpha, \beta)$-linear graph, which is not a tree,
then there exists an infinite family of (finite) connected 2-walk $(\alpha, \beta)$-linear graphs with unbounded diameter.
\begin{thm}\label{thm-sec3}
\begin{itemize}
  \item[\emph{(i)}] If $G$ is a connected graph, which is not a tree, with exactly two main eigenvalues, then it is always possible to construct
                   an infinite family of graphs with same two main eigenvalues and unbounded diameter.
  \item[\emph{(ii)}] Let $G$ be  a connected $\delta$-harmonic graph, which is not a tree, then there exists an infinite family of finite connected $\delta$-harmonic graphs with unbounded diameter.
\end{itemize}
\end{thm}
The following result of Mu-huo and Liu \cite{ML07} is an easy consequence of Theorem \ref{thm-sec3}.
\begin{prp}\cite{ML07}
Let $\delta \geq 3$ be an integer.
Assume there exists a connected $\delta$-harmonic graph with valencies $d_1 > d_2 > \ldots > d_t$.
There exists infinitely many pair-wise non-isomorphic connected $\delta$-harmonic graphs with distinct valencies
$d_1 > d_2 > \ldots > d_t$ if and only if $\{ d_1, \ldots, d_t \} \neq \{ 1, \delta, \delta^2 -\delta +1\}$.
\end{prp}


\begin{table}[h!]
\centering
\begin{tabular}{|c|c|c|c|}
  \hline
  $(\alpha,\beta)$ & main eigenvalues & valencies & number \\ \hline
  $(8,-9)$ & $4\pm\sqrt{7}$ & $3^{(1)},5^{(3)},7^{(12)}$ & 240 \\ \hline
  $(8,-9)$ & $4\pm\sqrt{7}$ & $5^{(6)},7^{(9)},9^{(1)}$ & 1120 \\ \hline
  $(8,-8)$ & $4\pm\sqrt{8}$ & $4^{(2)},6^{(8)},8^{(6)}$ & 2160 \\ \hline
  $(8,-5)$ & $4\pm\sqrt{11}$ & $3^{(1)},5^{(3)},7^{(8)},9^{(4)}$ & 2880 \\ \hline
  $(8,-5)$ & $4\pm\sqrt{11}$ & $5^{(6)},7^{(5)},9^{(5)}$ & 1152 \\ \hline
  $(8,-4)$ & $4\pm\sqrt{12}$ & $2^{(1)},6^{(6)},8^{(8)},10^{(1)}$ & 720 \\ \hline
  $(8,-4)$ & $4\pm\sqrt{12}$ & $4^{(4)},8^{(12)}$ & 240 \\ \hline
  $(8,-4)$ & $4\pm\sqrt{12}$ & $4^{(2)},6^{(6)},8^{(6)},10^{(2)}$ & 3360 \\ \hline
  $(8,-1)$ & $4\pm\sqrt{15}$ & $1^{(1)},7^{(9)},9^{(6)}$ & 240 \\ \hline
  $(8,-1)$ & $4\pm\sqrt{15}$ & $3^{(1)},5^{(2)},7^{(7)},9^{(5)},11^{(1)}$ & 2880 \\ \hline
  $(8,-1)$ & $4\pm\sqrt{15}$ & $5^{(5)},7^{(4)},9^{(6)},11^{(1)}$ & 1440 \\ \hline
  (8,0) & 8,~0 & $0^{(1)},8^{(15)}$ & 16 \\ \hline
  (8,0) & 8,~0 & $2^{(1)},6^{(4)},8^{(8)},10^{(3)}$ & 960 \\ \hline
  (8,0) & 8,~0 & $4^{(3)},8^{(12)},12^{(1)}$ & 240 \\ \hline
  (8,0) & 8,~0 & $4^{(2)},6^{(4)},8^{(6)},10^{(4)}$ & 2880 \\ \hline
  (8,0) & 8,~0 & $6^{(10)},10^{(6)}$ & 192 \\ \hline
  (8,3) & $4\pm\sqrt{19}$ & $3^{(1)},5^{(2)},7^{(3)},9^{(9)},11^{(1)}$ & 960 \\ \hline
  (8,3) & $4\pm\sqrt{19}$ & $3^{(1)},7^{(9)},9^{(3)},11^{(3)}$ & 320 \\ \hline
  (8,3) & $4\pm\sqrt{19}$ & $5^{(3)},7^{(6)},9^{(4)},11^{(3)}$ & 1920 \\ \hline
  (8,4) & $4\pm\sqrt{20}$ & $4^{(1)},6^{(4)},8^{(6)},10^{(4)},12^{(1)}$ & 2880 \\ \hline
  (8,4) & $4\pm\sqrt{20}$ & $6^{(8)},10^{(8)}$ & 180 \\ \hline
  (8,7) & $4\pm\sqrt{23}$ & $5^{(3)},7^{(6)},9^{(4)},11^{(3)}$ & 1920 \\ \hline
  (8,7) & $4\pm\sqrt{23}$ & $5^{(2)},7^{(4)},9^{(8)},11^{(1)},13^{(1)}$ & 720 \\ \hline
  (8,8) & $4\pm\sqrt{24}$ & $4^{(2)},8^{(6)},10^{(8)}$ & 240 \\ \hline
  (8,8) & $4\pm\sqrt{24}$ & $6^{(4)},8^{(6)},10^{(4)},12^{(2)}$ & 1440 \\ \hline
  (8,11) & $4\pm\sqrt{27}$ & $7^{(6)},9^{(6)},11^{(3)},13^{(1)}$ & 960 \\ \hline
  (8,12) & $4\pm\sqrt{28}$ & $6^{(2)},8^{(6)},10^{(6)},12^{(2)}$ & 480 \\ \hline
  (8,12) & $4\pm\sqrt{28}$ & $6^{(1)},8^{(8)},10^{(6)},14^{(1)}$ & 240 \\ \hline
  (8,12) & $4\pm\sqrt{28}$ & $8^{(12)},12^{(4)}$ & 80 \\ \hline
  (8,15) & $4\pm\sqrt{31}$ & $5^{(1)},9^{(10)},11^{(5)}$ & 96 \\ \hline
  (8,15) & $4\pm\sqrt{31}$ & $9^{(15)},15^{(1)}$ & 16 \\ \hline
\end{tabular}\label{tab1}
\vspace{0.4cm}
\caption{The non-regular graphs in the switching class of the symplectic two-graph on 16 vertices.}
\end{table}

\end{document}